\documentclass [reqno] {amsart}

\usepackage{amsfonts}
\usepackage{amssymb}

\usepackage{color}
\usepackage{hyperref}

\newtheorem{theorem}{Theorem}[section]
\newtheorem{lemma}[theorem]{Lemma}
\newtheorem{corollary}[theorem]{Corollary}

\theoremstyle{definition}

\theoremstyle{remark}

\numberwithin{equation}{section}

\newcommand{\B}{\mathbf{B}_q}

\begin{document}
	
	\title[Existence theorem]{Existence theorem of a weak solution for Navier-Stokes type equations associated with de~Rham complex
	}

	\author[A. Polkovnikov]{Alexander Polkovnikov}
	
	\address{Siberian Federal University,
		Institute of Mathematics and Computer Science,
		pr. Svobodnyi 79,
		660041 Krasnoyarsk,
		Russia}
	
	\email{paskaattt@yandex.ru}

	
	
	\begin{abstract}
		Let $ \{d_q, \Lambda^{q} \} $ be de Rham complex on a smooth compact closed manifold $X$ over $ \mathbb{R}^3 $ with Laplacians $\Delta_{q} $. We consider operator equations, associated with the	parabolic differential operators $\partial_t + \Delta_2 + N^{2} $ on the second step of complex with nonlinear bi-differential operator  of zero order $ N^{2} $. Using by projection on the next step of complex we show that the equation has unique solution in special 
		Bochner-Sobolev type functional spaces for some (small enough) time $ T^* $.
	\end{abstract}

	\maketitle
	
	\section*{Introduction}
	Consider the de Rham complex on a Riemannian $ n $-dimensional smooth compact closed manifold $X$ with vector bundles $ \Lambda^{q} $ of exterior forms of degree $ q $ over $X$,
	\begin{equation}\label{key}
		0\longrightarrow \varOmega^0(X)\xrightarrow{d_0} \varOmega^1(X)
		\xrightarrow{d_1} \cdots \xrightarrow{d_{n-1}} \varOmega^n(X)\longrightarrow 0.
	\end{equation}
Here $ \varOmega_q(X) $ denotes the space of all differential forms of degree $ q $ with smooth coefficients on $ X $.
	In this case the Laplacians	$  \Delta_{q} = d_q^* d_{q} + d_{q-1} d_{q-1}^*$, 	$ q=0, 1, \dots, n $, of the complex are second order strongly elliptic differential operators on $X$, where operator $ d_q^* $ is a formal adjoint to $ d_q $. As usual, for $ q<0 $ and $ q\geq n $ we assume that $ 	d_q = 0 $.
	
	We want to study the non-linear problems, associated with the complex.
	With this purpose, we denote by $ M_{i,j} $ two bilinear bi-differential operators of zero order (see \cite{Bourbaki} or \cite{Tarkhanov95}),
	\begin{equation}\label{eq1}
		\begin{array}{c}
			M_{q,1}(\cdot,\cdot): \big( \varOmega^{q+1}(X),\ \varOmega^{q}(X)\big) \to \varOmega^{q}(X),\\ 
			M_{q,2}(\cdot,\cdot): \big( \varOmega^{q}(X),\ \varOmega^{q}(X)\big) \to \varOmega^{q-1}(X).
		\end{array}
	\end{equation}
	We set for a differential form $ u $ of the degree  $ {q} $
	\begin{equation}\label{eq2}
		N^{q}(u) =: M_{q,1} ( d_{q}u, u) +  d_{q-1} M_{q,2}(u, u).
	\end{equation}
Note, that operator $ N^{q}(u) $ is non-linear.
	
	Let now time $ T>0 $ is finite. Then  for any fixed positive number $ \mu $ the operators $ \partial_t + \mu \Delta_q $ are parabolic on the cylinder $  X \times (0,T) $ (see \cite{Eidelman}). Consider the following initial problem: given sufficiently regular differential forms $ f $ of the induced bundle $ \Lambda^{q} (t) $ (the variable $ t $ enters into this bundle as a 
	parameter) and $ u_0 $ of the bundle $ \Lambda^{q} $, find a differential forms $ u $ of the induced bundle $ \Lambda^{q} (t) $ and $ p $ of the induced bundle $ \Lambda^{q-1} (t) $ such that
	\begin{equation}\label{eq3}
		\begin{cases}
			\partial_t u + \mu \Delta_{q} u + N^{q}(u) + d_{q-1} p = f& \text{in } X \times (0,T),\\
			d_{q-1}^* u =0 & \text{in } X \times [0,T],\\
			d_{q-2}^* p =0 & \text{in } X \times [0,T],\\
			u(x,0) = u_0& \text{in } X,
		\end{cases}
	\end{equation}
	
	For general elliptic complexes this problem was considered in the works \cite{ShlapunovPar} and \cite{N28B}, where the open mapping theorems were proved in the special spaces of H\"older (see \cite{ShlapunovPar}) and Sobolev (see \cite{N28B}) types. It means that the range of the non-linear operator ${\mathcal A}_q$ , related to the problem, is open in the constructed spaces. However, obtaining an existence theorem for the solution (even the 
	so-called weak one) and closedness of the range of the related non-linear 
	operator in such spaces appears to be a more difficult task.
	
	For example, if we take $q=1$ and a suitable nonlinear term we may treat \eqref{eq3} as the initial problem for the well known Navier-Stokes equations for incompressible fluid over the manifold $X$ (see, for instance, \cite{Tarkhanov19} or \cite{ShTar21}). Note that the equation with respect to $p$ is actually missing in this case, because $d_{-1}^* = 0$.
	
	We consider problem \eqref{eq3} in the case $n=3$, $q=2$ and a special nonlinearity $  M_{q,1} ( d_{q}u, u) = (d_{q}u) u $. It easy to see that in this case we can treat the de Rham differentials as $ d_{2} = \mathrm{div} $, $ d_{1} = \mathrm{rot} $, $ d_{2}^* = - \nabla $, $ d_{1}^* = \mathrm{rot} $ and then \eqref{eq3} transforms to
	\begin{equation}\label{eq34}
		\begin{cases}
			\partial_t u + \mu \Delta_{2} u + N^{2}(u) + \mathrm{rot}\, p = f& \text{in } X \times (0,T),\\
			\mathrm{rot}\, u =0 & \text{in } X \times [0,T],\\
			\mathrm{div}\, p =0 & \text{in } X \times [0,T],\\
			u(x,0) = u_0& \text{in } X,
		\end{cases}
	\end{equation}
	where 
	\begin{equation}\label{eq35}
		N^{2}(u) = (\mathrm{div}\,u) u +  \mathrm{rot}\,( M_{q,2}(u, u)),
	\end{equation}	
	and Laplacian
	\[
	\Delta_{2} u = d_2^* d_{2} + d_{1} d_{1}^* = -\nabla  \mathrm{div}\, u + \mathrm{rot}\,\mathrm{rot}\,u = - \Delta u.
	\]
	Here $ \Delta u$ is a standard Laplace operator applied componentwise to the differential form $ u $ in the space variable $ x $.

Using projection to the next step of complex \eqref{key}, we prove an existence theorem of weak (distributional) solution in the constructed Bochner-Sobolev type spaces for some 
	(small enough) time 	$ T^* $. 
	Note, that considering general non-linear perturbations 
	of linear parbolic equations one have to impose essential restrictions 
	on the non-linear term $ N^2(u) $  in order to achieve existence of  
	weak solutions. For example, one of such condition can be positiveness of non-linear operator $ N^2(u) $. However, we do not impose such strong conditions on the non-linear term, but still have an existence of weak solutions due to special properties 
	of the de Rham complex.

	\section{Functional spaces}

	Denote by $ L^p_{\Lambda^{q}} $, $ 1\leq p \leq\infty $, space of differential forms of the degree  $ {q} $ with coefficients in the Lebesgue space $ L^p (X) $. 
	In a similar way we designate the spaces of forms on $ X $ whose components are of	Sobolev class or have continuous partial derivatives. We denote it by $ W^{s,p}_{\Lambda^{q}} $ and $ C^{s}_{\Lambda^{q}} $ respectively with smoothness $ s $. In particular case, for $ p=2 $ we designate $ H^{s}_{\Lambda^{q}}:= W^{s,2}_{\Lambda^{q}}$.
	
	For calculations, it is convenient to use the fractional powers of the Laplace operator. Namely, for differential form $ u $ of degree $ q $ we denote by 
	\begin{equation}\label{frac.laplase}
		\nabla^m_q u : = 
		\begin{cases}
			\Delta_{q}^{m/2}u,\quad &m\ \text{is even},\\
			(d_q \oplus d_{q-1}^*)\Delta_{q}^{(m-1)/2}u,\quad &m\ \text{is odd}.\\
		\end{cases}
	\end{equation}
It easy to see that integration by parts yields
\[
\sum_{|\alpha| = m} \|\partial^{\alpha} u \|^2_{L^2_{\Lambda^{q}}} = \|\nabla^m_q u\|^2_{L^2_{\Lambda^{q}}}.
 \]

	Now, we want to recall the standard Hodge theorem for elliptic complexes. For this purpose denote by $ \mathcal{H}^q $ the harmonic space of
	complex (\ref{key}), i.e.
	\begin{equation}\label{eq20}
		\mathcal{H}^q = \left\lbrace u \in C^{\infty}_{\Lambda^{q}}: d_q	u = 0 \mbox{ and } d_{q-1}^*	u = 0\mbox{ in }X \right\rbrace,
	\end{equation}
	and by $ \Pi^i $ the orthogonal projection from $ L^2_{\Lambda^{q}} $ onto $ \mathcal{H}^q $. 
	
	\begin{theorem}\label{parametrix.laplas}
		Let $ 0 \leq q \leq n $, $ s\in \mathbb{Z}_+  $. Then operator 
		\begin{equation}\label{lapl}
			\Delta_{q}: H^{s+2}_{\Lambda^{q}}\to H^s_{\Lambda^{q}}
		\end{equation}
		is Fredholm:
		
		(1) the kernel of operator (\ref{lapl}) equals to the finite-dimensional space $ \mathcal{H}^q $;
		
		(2) given $ v\in H^s_{\Lambda^{q}} $ there is a form $ u\in H^{s+2}_{\Lambda^{q}} $  such that $ \Delta_{q} u = v  $ if and
		only if $ (v, h)_{L^2_{\Lambda^{q}}} = 0 $ for all $ h \in \mathcal{H}^q $;
		
		(3) there exists a pseudo-differential operator $ \varphi^i $ on $ X $ such that the operator
		\begin{equation}
			\varphi^q: H^s_{\Lambda^{q}}\to H^{s+2}_{\Lambda^{q}},
		\end{equation}
		induced by $ \varphi^q $, is linear bounded and with the identity $ I $ we have
		\begin{equation}\label{eq22}
			\varphi^q \Delta_{q} = I - \Pi^q\mbox{ on } H^{s+2}_{\Lambda^{q}},\quad  \Delta_{q} \varphi^q = I - \Pi^q\mbox{ on } H^{s}_{\Lambda^{q}}
		\end{equation}
	\end{theorem}
	\begin{proof}
		See, for instance, 
		\cite[Theorem 2.2.2]{Tarkhanov95}.
	\end{proof}

	Denote by $V^s_{\Lambda^{q}} := H^{s}_{\Lambda^{q}} \cap S_{d_{q-1}^*} $ the space of all differential forms $ u\in H^{s}_{\Lambda^{q}} $ satisfying $ d_{q-1}^* u = 0 $ in the sense of distributions in $ X $. Let now $L^2(I,H^{s}_{\Lambda^{q}})$ be the Bochner space of $L^2$-mappings
	\[
	u(t): I \to H^{s}_{\Lambda^{q}},
	\]
	where $ I = [0,T] $, see, for instance, \cite{Lion69}.
	It is a Banach space with the norm 
	\[
	\|u\|_{L^2(I,H^{s}_{\Lambda^{q}})}^2 = \int_0^T \|u\|^2_{H^{s}_{\Lambda^{q}}} dt.
	\]
	We need to introduce suitable Bocner-Sobolev type spaces, see \cite{ShT20} or \cite{ShTar21} for the de Rham complex and \cite{N28B} for the general elliptic complexes. 
	Namely, for $s\in\mathbb{Z}_+ $ denote by $ B^{k,2s,s}_{q,\text{vel}} (X_T) $ the 
	space of all differential forms of degree $ q $ over $ X_T:= X\times [0,T] $ 
	with variable $ t\in [0,T] $ as a parameter,
	such that
	\[
	u\in C (I,V_{\Lambda^{q}}^{k+2s})\cap L^2(I,V_{\Lambda^{q}}^{k+2s+1})
	\]
	and 
	\[
	\nabla^m_q \partial^j_t u \in C (I,V_{\Lambda^{q}}^{k+2s-m-2j})\cap L^2(I,V_{\Lambda^{q}}^{k+2s+1-m-2j})
	\]
	for all $ m + 2j \leq 2s$. It is a Banach space with the norm 
	$$
	\| u \|^2_{B^{k,2s,s}_{q,\text{vel}}}
	:=
	\sum_{m+2j \leq 2s \atop 0\leq l \leq k}
	\| \nabla^l_q \nabla^m_q \partial_t^j u \|^2
	_{C (I,{L}_{\Lambda^{q}}^2)}
	+ \| \nabla^{l+1}_q \nabla^m_q \partial_t^j u \|^2
	_{L^2 (I,{L}_{\Lambda^{q}}^2 )}.
	$$
	
	Similarly, for $s, k \in {\mathbb Z}_+$, we define the space 
	$B^{k,2s,s}_{q,\mathrm{for} } (X_T)$ to consist of all differential forms
	\[
	f\in C (I,{H}_{\Lambda^{q}}^{2s+k}) \cap L^2 (I,{H}_{\Lambda^{q}}^{2s+k+1})
	\]
	with the property that 
	$$
	\nabla^m_q \partial _t^j f
	\in
	C (I,{H}_{\Lambda^{q}}^{k+2s-m-2j}) \cap L^2 (I,{H}_{\Lambda^{q}}^{k+2s-m-2j+1})
	$$
	for all
	$
	m+2j \leq 2s.
	$
	We endow the space $B^{k,2s,s}_{q,\mathrm{for}} (X_T)$ 
	with the natural norm 
	$$
	\| f \|^2_{B^{k,2s,s}_{q,\mathrm{for}}}
	:=
	\sum_{m+2j \leq 2s \atop 0\leq l \leq k}
	\| \nabla^l_q \nabla^m_q \partial_t^j f \|^2
	_{C (I,{L}_{\Lambda^{q}}^2)}
	+ \| \nabla^{l+1}_q \nabla^m_q \partial_t^j f \|^2
	_{L^2 (I,{L}_{\Lambda^{q}}^2)}. 
	$$
	
	Lastly, the space for the differential form $p$ we denote by 
	$B^{k+1,2s,s}_{q-1,\mathrm{pre}}(X_T)$.  
	This space consists of all forms $p$ from the space 
	$C (I,H_{\Lambda^{q-1}}^{2s+k+1}) \cap 
	L^2 (I,H_{\Lambda^{q-1}}^{2s+k+2} )$ such that $d_{q-1} p \in B^{k,2s,s}_{q,\mathrm{for}}(X_T)$, $ d_{q-2}^* p = 0 $ and for all $h \in \mathcal{H}^{q-1} $
	\begin{equation}\label{eq.p.ort.0}
		( p,h)_{L^2_{\Lambda^{q-1}}} = 0.
	\end{equation}
It is a Banach space with the norm 
	$$
	\| p \|_{B^{k+1,2s,s}_{q-1,\mathrm{pre}}}
	= \| d_{q-1} p \|_{B^{k,2s,s}_{q,\mathrm{for}}}.
	$$
	
	Define now for suitable forms $ v $ and $ w $ of degree $ q $ a bi-differential operator
	\begin{equation}
		\B (w,v) = M_{q,1} ( d_{q}w, v) + M_{q,1} ( d_q v, w) +  d_{q-1}\big(  M_{q,2}(w, v)  +  M_{q,2}(v, w)\big),
	\end{equation}
	with the operators $ M_{q,1} $ and $ M_{q,2} $ satisfying 
	\begin{equation}\label{eq9}
		|M_{q,1}(u,v)|\leq c_{q,1} |u|\,|v|,\quad |M_{q,2}(u,v)|\leq c_{q,2} |u|\,|v| 
		\mbox{ on } X 
	\end{equation}
	with some  positive constants $ c_{i,j} $.
	Following theorem allows us to see the correctness of operators in this spaces.
	\begin{theorem}
		\label{l.NS.cont.0}
		Suppose that
		$s \in \mathbb N$,
		$k \in {\mathbb Z}_+$ and $2s+k>\frac{n}{2}-1$.
		Then the mappings
		$$
		\begin{array}{rrcl}
			\nabla^m_q :
			& B^{k,2(s-1),s-1}_{{q},\mathrm{for}}(X_T)
			& \to
			& B^{k-m,2(s-1),s-1}_{q,\mathrm{for}}(X_T),\ m\leq k
			\\[.2cm]
			\Delta_q :
			& B^{k,2s,s}_{q,\mathrm{vel}} (X_T)
			& \to
			& B^{k,2(s-1),s-1}_{q,\mathrm{for}}(X_T),
			\\[.2cm]
			\partial_t :
			& B^{k,2s,s}_{q,\mathrm{vel}}(X_T)
			& \to
			& B^{k,2(s-1),s-1}_{q,\mathrm{for}}(X_T),\\
		\end{array}
		$$
		are continuous.
		Besides, if 
		$w, v \in B^{k+2,2(s-1),s-1}_{{i},\mathrm{vel}}(X_T)$
		then the mappings
		
		\begin{equation}\label{eq8}
			\begin{array}{rrcl}
				\B(w,\cdot) :
				& B^{k+2,2(s-1),s-1}_{q,\mathrm{vel}}(X_T) 
				& \to
				& B^{k,2(s-1),s-1}_{q,\mathrm{for}}(X_T),
				\\
				\B (w,\cdot) :
				& B^{k,2s,s}_{q,\mathrm{vel}}(X_T)
				& \to
				& B^{k,2(s-1),s-1}_{q,\mathrm{for}}(X_T),\\
			\end{array}
		\end{equation}
		are continuous, too. In particular, for all $w, v \in B^{k+2,2(s-1),s-1}_{q, \mathrm{vel}}(X_T)$ there is positive constant $c_{s,k} $, independent on $ v $ and $w$, such that  
		\begin{equation}\label{eq.B.pos.bound}
			\|  \B (w,v)\|_{B^{k,2(s-1),s-1}_{q,\mathrm{for}}}
			\leq c_{s,k} 
			\|w\|_{B^{k+2,2(s-1),s-1}_{q,\mathrm{vel}} } 
			\|v\|_{B^{k+2,2(s-1),s-1}_{q,\mathrm{vel}} }.
		\end{equation}
	\end{theorem}
	
	\begin{proof}
		See, for instance, \cite{ShTar21} or \cite{N28B}.
	\end{proof}
	
	Let us introduce now the Helmholtz type projection 
	$\mathrm {P}^q$ from $ B^{k,2(s-1),s-1}_{q,\mathrm {for}} (X_T)$ to the kernel of operator $ d_q^* $.
	
	\begin{lemma}\label{proector} 
		If $ s $, $ k \in \mathbb{Z}_+ $, then for each $q$ the pseudo-differential operator $ \mathrm {P}^q = d_q^* d_q \varphi^q + \Pi^q$ on $ X $ induce continuous map
		\begin{equation}\label{cont.proetor}
			\mathrm {P}^q: B^{k,2(s-1),s-1}_{q,\mathrm {for}} (X_T)\to B^{k,2(s-1),s-1}_{q, \mathrm {vel}} (X_T),
		\end{equation}
		such that
		\begin{equation*}
			\mathrm {P}^q\circ \mathrm {P}^q u = \mathrm {P}^q u,\quad
			(\mathrm {P}^q u, v)_{L^2_{\Lambda^q}(X)} = (u, \mathrm {P}^q v)_{L^2_{\Lambda^q}(X)},\quad
			(\mathrm {P}^q u, (I-\mathrm {P}^q) u)_{L^2_{\Lambda^q}(X)} = 0
		\end{equation*}
		for all $ u,v\in B^{k,2(s-1),s-1}_{q,\mathrm {for}} $.
	\end{lemma}
	\begin{proof}
		See, for instance, \cite{N28B}.
	\end{proof}
	The following Lemma is just a consequence of Hodge Theorem \ref{parametrix.laplas}.
	\begin{lemma}\label{p.nabla.Bochner} 
		Let $F\in B^{k,2(s-1),s-1}_{q, \mathrm {for}} (X_T)$ 
		satisfy $\mathrm {P}^q F=0$ in $X_T$.  
		Then there is a unique section 
		$p \in B^{k+1,2(s-1),s-1}_{q-1,\mathrm {pre}}(X_T)$
		such that (\ref{eq.p.ort.0}) holds
		and 
		\begin{equation}\label{eq.nabla.Bochner}
			d_{q-1} p = F \mbox{ in } X \times [0,T].
		\end{equation}
	\end{lemma}
	Now we are ready to go to the main section of this paper.
	
	\section{Existence theorem}
	
	In order to get existence theorem to the Problem \eqref{eq34} we use a projection to the next step of complex \eqref{key}.
	Namely, applying an operator $ d_2 = \mathrm{div} $ to the equation \eqref{eq34} we have
	\begin{equation}\label{eq37}
		\begin{cases}
			\partial_t \mathrm{div} u - \mu\, \mathrm{div}(\nabla  \mathrm{div}\, u) + \mathrm{div}((\mathrm{div}\,u) u) = \mathrm{div}f& \text{in } X \times (0,T),\\
			\mathrm{div}\, u(x,0) = \mathrm{div}\, u_0& \text{in } X,
		\end{cases}
	\end{equation}
because of $ \mathrm{rot}\, u =0 $ and $ \mathrm{div} \circ \mathrm{rot} \equiv 0 $.
Now, 
\[
\mathrm{div}((\mathrm{div}\,u) u) = (\mathrm{div}\,u)^2 + \Delta u\cdot u = (\mathrm{div}\,u)^2 + \nabla  \mathrm{div}\, u\cdot u.
\]
By Theorem \ref{parametrix.laplas} 
\[
u = 	\varphi^2 \Delta_{2} u + \Pi^2 u = \varphi^2 \nabla  \mathrm{div}\, u + \Pi^2 u.
\]
	Denote
	\[
	g= \mathrm{div}\, u,
	\]
	then we can rewrite \eqref{eq37} by the next way
	\begin{equation}\label{eq38}
		\begin{cases}
			\partial_t g - \mu\, \mathrm{div}(\nabla  g) + g^2 + \nabla  g \cdot (\varphi^2 \nabla  g + \Pi^2 u) = \mathrm{div}f& \text{in } X \times (0,T),\\
			g(x,0) = \mathrm{div}\, u_0& \text{in } X.
		\end{cases}
	\end{equation}

	\begin{theorem}
		\label{t.weak.g}
		Given any pair $(f,u_0) \in L^2 (I,(V^0_{\Lambda^{2}})') \times V^1_{\Lambda^{2}}$. There is time $ t_0\in (0,T] $ such that for all $ t\in [0,t_0] $ there exist a differential form
		$g \in C (I,L^2_{\Lambda^{3}}) \cap L^2 (I,H_{\Lambda^{3}}^1)$ with
		$\partial_t g \in L^2 (I,(H_{\Lambda^{3}}^1)')$,
		satisfying
		\begin{equation}
			\label{eq.weak_g}
			\left\{
			\begin{array}{rcl}
				\displaystyle
				\frac{d}{dt} (g,v)_{{L}_{{\Lambda^{3}}}^2}
				+ \mu (\nabla g, \nabla v)
				_{{L}_{{\Lambda^{2}}}^2}
				& =
				& \langle \mathrm{div}\,f - g^2 -\nabla g\cdot(\varphi^2\,\nabla g + \Pi^2 u), v \rangle,
				\\
				g (\cdot,0)
				& =
				& \mathrm{div}\, u_0
			\end{array}
			\right.
		\end{equation}
		for all $v \in H_{\Lambda^{3}}^k$ with $ k\geq 2 $. 
	\end{theorem}
	
	\begin{proof}
		Let 
		$\{u_m\}$ be the sequence of Faedo-Galerkin approximations, 
		namely, 
		\begin{equation}\label{eq36}
			u_m = \sum_{j=1}^M c^{(m)}_{j} (t) b_j (x),
		\end{equation}
	then
	\begin{equation}\label{faedo}
		g_m= \mathrm{div}\, u_m = \sum_{j=1}^M c^{(m)}_{j} (t)\, \mathrm{div}\,b_j (x),
	\end{equation}
		where the system $\{ b_j \}_{j \in {\mathbb N}}$ 
		is a ${L}_{\Lambda^{2}}^2 (X)$-orthogonal basis in $V^1_{\Lambda^{2}}$ and the functions $u_{m} $ satisfy the following relations 
		\begin{equation} \label{eq_edin_1}
			\frac{d}{dt} (g_m,\mathrm{div}\, b_j)_{{L}_{{\Lambda^{3}}}^2}
			+ \mu (\nabla g_m, \nabla \mathrm{div}\, b_j)
			_{{L}_{{\Lambda^{2}}}^2}
			 = \end{equation}
		 \[
			\langle \mathrm{div}\,f - g_m^2 -\nabla g_m\cdot\varphi^2\,\nabla g_m - \nabla g_m\cdot \Pi^2 u_m, \mathrm{div}\, b_j \rangle,
		\]
		\begin{equation*}
			g_m(x,0) = \mathrm{div}\, u_{0,m} (x),
		\end{equation*}
		for all $0\leq j \leq m$ with the initial date $u_{0,m}$ from 
		the linear span ${\mathcal L} (\{b_j\}_{j=1}^m)$ such that the sequence 
		$\{u_{0,m}\}$ converges to $u_0$ in $V_{\Lambda^{2}}^1$. 
		For instance, as $\{u_{0,m}\}$ we may take  the orthogonal projection 
		onto the linear span ${\mathcal L} (\{b_j\}_{j=1}^m)$.
		
		Multiplying \eqref{eq_edin_1} by $ c_j^{(m)}(t) $ and summing by $ j $ we have
		\begin{equation} \label{eq_edin_2}
			 (\partial_t g_m,g_m)_{L_{\Lambda^{3}}^2}
			+ \mu (\nabla g_m, \nabla g_m)
			_{{L}_{{\Lambda^{2}}}^2}
			=
			\langle \mathrm{div}\,f - g_m^2 -\nabla g_m\cdot\varphi^2\,\nabla g_m - \nabla g_m\cdot \Pi^2 u_m, g_m \rangle.
		\end{equation}
It follows from Lemma by J.-L. Lions (see, for instance, \cite[Ch.~III, \S~1, Lemma~1.2]{Temam79}) that
\begin{equation*}
	\frac{d}{dt} \| g_m (\cdot, t) \|^2_{L_{\Lambda^{3}}^2}
	= 2\, \langle \partial_t g_m, g_m \rangle.
\end{equation*}
	Then integrating by  $t \in [0,T]$ we see that
		\begin{equation} \label{eq_edin_3}
			\| g_m(\cdot, t)\|^2_{{L}_{\Lambda^{3}}^2}
			+ 2\mu\int_0^t \|\nabla g_m\|^2_{{L}_{{\Lambda^{2}}}^2}dt
			=
				\end{equation}
			\[
\|g_m(\cdot, 0 )\|^2_{{L}_{\Lambda^{3}}^2} + 	2\int_0^t \langle \mathrm{div}\,f - g_m^2 -\nabla g_m\cdot\varphi^2\,\nabla g_m - \nabla g_m\cdot \Pi^2 u_m, g_m \rangle dt.
	\]
		
		Since $ f\in L^2 (I,L^2_{\Lambda^{2}}) $ then $ \mathrm{div}\,f\in L^2 (I,(V^1_{\Lambda^{3}})') $ and
		\begin{equation}\label{eq_edin_4}
		2	\left| \int_0^t \langle \mathrm{div}\,f, g_m \rangle dt \right| \leq  2\int_0^t \| \mathrm{div}\,f\|_{(V^1_{\Lambda^{3}})'} \| g_m \|_{V^1_{\Lambda^{3}}} dt \leq
		\end{equation}
		\[
		\frac{4}{\mu}\int_0^t\|\mathrm{div}\,f\|^2_{(V^1_{\Lambda^{3}})'}dt + \frac{\mu}{4} \int_0^t \| \nabla g_m \|^2_{L^2_{\Lambda^{2}}} dt + \frac{\mu}{4} \int_0^t \| g_m \|^2_{L^2_{\Lambda^{3}}} dt.
		\]
		
		On the other hand
		\begin{equation}\label{eq_edin_5}
			2	\left| \int_0^t \langle g_m^2, g_m \rangle dt\right| \leq 
			2	\int_0^t \|g_m\|^3_{L^3_{\Lambda^{3}}}dt.
		\end{equation}
Note that in our case $ \nabla_3 = -\nabla$ with $ n=3 $. Then from the Gagliardo-Nirenberg inequality (see \cite{Nir59} or \cite[Theorem 3.70]{Aubin82}) 
we have
	\begin{equation}\label{eq_edin_7}
		2\int_0^t \|g_m\|^3_{L^3_{\Lambda^{3}}}dt \leq 
	\end{equation}
\[
c\int_0^t \left[ \Big( \|\nabla g_m\|_{L^2_{\Lambda^{2}}} + \| g_m\|_{L^2_{\Lambda^{3}}}\Big)^{\frac{1}{2}}\| g_m\|^{\frac{1}{2}}_{L^2_{\Lambda^{3}}} + \| g_m\|_{L^2_{\Lambda^{3}}} \right]^3dt \leq
	\]
	\[
 c_1\int_0^t \left[ \|\nabla g_m\|^{\frac{1}{2}}_{L^2_{\Lambda^{2}}}\| g_m\|^{\frac{1}{2}}_{L^2_{\Lambda^{3}}} + \| g_m\|_{L^2_{\Lambda^{3}}} \right]^3 dt \leq
	\]
	\[
	c_2\int_0^t \left(  \|\nabla g_m\|^{\frac{3}{2}}_{L^2_{\Lambda^{2}}}\| g_m\|^{\frac{3}{2}}_{L^2_{\Lambda^{3}}} + \| g_m\|^{3}_{L^2_{\Lambda^{3}}} \right) dt\leq   
	\]
	\[
	\frac{\mu}{2}\int_0^t  \|\nabla g_m\|^{2}_{L^2_{\Lambda^{2}}}dt + c_3\int_0^t \left( \|g_m\|^{3}_{L^2_{\Lambda^{3}}} + \|g_m\|^{6}_{L^2_{\Lambda^{3}}}\right) dt
	\]
	with positive constants $ c $,  $ c_1 $ and $ c_2 $. The last expression is consequence of standard Young's inequality. Moreover, 
	there are positive constants $ c $ and $ c_1 $ such that
	\[
	\int_0^t \left( \|g_m\|^{3}_{L^2_{\Lambda^{3}}} + \|g_m\|^{6}_{L^2_{\Lambda^{3}}}\right) dt\leq 
	c\int_0^t \|g_m\|^{2}_{L^2_{\Lambda^{3}}}\left( 1 + \|g_m\|_{L^2_{\Lambda^{3}}}\right)^4 dt \leq
	\]
	\[
	 c_1 \left(\int_0^t \|g_m\|^{2}_{L^2_{\Lambda^{3}}}dt + \int_0^t \|g_m\|^6_{L^2_{\Lambda^{3}}} dt\right).
	\]
	Then we conclude that 
		\begin{equation}\label{eq_edin_12}
		2\int_0^t \|g_m\|^3_{L^3_{\Lambda^{3}}}dt \leq \frac{\mu}{2}\int_0^t  \|\nabla g_m\|^{2}_{L^2_{\Lambda^{2}}}dt + c \int_0^t \|g_m\|^{2}_{L^2_{\Lambda^{3}}}dt + c\int_0^t \|g_m\|^6_{L^2_{\Lambda^{3}}} dt
	\end{equation}
	with some constant $ c>0 $.
	Next,
		\[
			 \int_0^t \langle \nabla g_m\cdot\varphi^2\,\nabla g_m, g_m \rangle dt = \sum_{j=1}^3 \int_0^t \int\limits_X \partial_j g_m (\varphi^2 \partial_j g_m) g_m\,dx\,dt = 
		\]
\[
	 - \sum_{j=1}^3 \int_0^t \int\limits_X g_m (\varphi^2 \partial_j g_m) \partial_j g_m\,dx\,dt -   \int_0^t \int\limits_X g_m^3\,dx\,dt,
\]
	because $ \varphi^2\Delta g_m = g_m $. It means that
\[
	 \int_0^t \langle \nabla g_m\cdot\varphi^2\,\nabla g_m, g_m \rangle dt = -\frac{1}{2} \int_0^t \int\limits_X g_m^3\,dx\,dt,
\]
	and hence
\begin{equation}\label{eq_edin_8}
	2\left| \int_0^t \langle \nabla g_m\cdot\varphi^2\,\nabla g_m, g_m \rangle dt\right| \leq 	\int_0^t \|g_m\|^3_{L^3_{\Lambda^{3}}}dt.
\end{equation}
	
	Finally,
	\[
		\int_0^t \langle \nabla g_m \cdot \Pi^2 u_m, g_m \rangle dt = \sum_{j=1}^3 \int_0^t \int\limits_X \partial_j g_m (\Pi^2 u_m^j) g_m\,dx\,dt = 
\]
	\[
	 - \sum_{j=1}^3 \int_0^t \int\limits_X  g_m (\Pi^2 u_m^j) \partial_j g_m\,dx\,dt - \sum_{j=1}^3 \int_0^t \int\limits_X  g^2_m \partial_j(\Pi^2 u_m^j) \,dx\,dt,
	\]
and then
\begin{equation}\label{eq_edin_9}
	\int_0^t \langle \nabla g_m \cdot \Pi^2 u_m, g_m \rangle dt = 0,
	\end{equation}
	because $ \mathrm{div}\, \Pi^2 u_m = 0 $.
	
	Now, inequalities \eqref{eq_edin_3} - \eqref{eq_edin_9} give
	\begin{equation}\label{eq_edin_10}
		\| g_m(\cdot, t)\|^2_{L^2_{\Lambda^{3}}} + 2\mu\int_0^t \|\nabla g_m\|^2_{L^2_{\Lambda^{3}}}dt \leq \|g_m(\cdot, 0 )\|^2_{L^2_{\Lambda^{3}}} + 
	\end{equation}
		\[
			\frac{4}{\mu}\int_0^t\|\mathrm{div}\,f\|^2_{(V^1_{\Lambda^{3}})'}dt + \mu \int_0^t \| \nabla g_m \|^2_{L^2 _{\Lambda^{2}}} + \frac{\mu}{4} \int_0^t \| g_m \|^2_{L^2 _{\Lambda^{3}}} dt + 
		\]
\[
 2  c \int_0^t \|g_m\|^{2}_{L^2_{\Lambda^{3}}}dt + 2 c\int_0^t \|g_m\|^6_{L^2_{\Lambda^{3}}} dt,
\]
and then 
	\begin{equation}\label{eq_edin_11}
	\| g_m(\cdot, t)\|^2_{L^2_{\Lambda^{3}}} + \mu\int_0^t \|\nabla g_m\|^2_{L^2_{\Lambda^{3}}}dt\leq 
\|g_m(\cdot, 0 )\|^2_{{L}^2_{\Lambda^{3}}} +  
\end{equation}
\[
\frac{4}{\mu}\|\mathrm{div}\,f\|^2_{L^2\left(I,{( V^1_{\Lambda^{3}})'}\right) } + \left( \frac{\mu}{4} + 2c\right)  \int_0^t \| g_m \|^2_{L^2_{\Lambda^{3}}} dt + 2 c\int_0^t \|g_m\|^6_{L^2_{\Lambda^{3}}} dt.
\]

It follows from the Gronwall-Perov's Lemma (see, for instance \cite[p.~360]{MPF91}) that there exist a time $ t_0\in (0,T] $ and a positive constant $ C_{t_0} $ such that
	\begin{equation}\label{eq_edin_13}
	\| g_m(\cdot, t)\|^2_{L^2_{\Lambda^{3}}} \leq C_{t_0}
\end{equation}
for all $ t\in[0,t_0] $. Then the sequence $ g_m $ is bounded in $ L^\infty(I_{t_0},L^2_{\Lambda^{3}}) $, where $ I_{t_0} = [0,t_0] $. Moreover it follows from \eqref{eq_edin_11} and \eqref{eq_edin_13} that 
$ \| \nabla g_m(\cdot, t)\|^2_{L^2(I_{t_0},L^2_{\Lambda^{3}})} $ is bounded too. 
It means that there is a subsequence that converges weakly-$ * $ in $ L^\infty(I_{t_0},L^2_{\Lambda^{3}}) $ and weakly in $ L^2(I_{t_0},H^1_{\Lambda^{3}}) $ to some $ g\in L^\infty(I_{t_0},L^2_{\Lambda^{3}}) \cap L^2(I_{t_0},H^1_{\Lambda^{3}}) $. We use the same designation $ g_m $ for such a subsequence. Then the standard argument show (see, for instance, \cite{LiMa72}, \cite{Temam79} or \cite{Ladyz}) that we can to pass to the limit in \eqref{eq_edin_1} 
with respect to $ m\to \infty $ and to conclude that the element $ g $ satisfies \eqref{eq.weak_g}.

	\end{proof}
	
	Let us now return to the Problem \eqref{eq34}. Denoting again $ g = \mathrm{div}\,u $ and multiplying \eqref{eq34} scalar in $ L^2_{\Lambda^{2}} $ by differential form $ v\in V^k_{\Lambda^{2}} $ we get
	\begin{equation}
		\label{eq.NS.lin.weak_u}
		\left\{
		\begin{array}{rcl}
			\displaystyle
			\frac{d}{dt} (u,v)_{{L}_{{\Lambda^{2}}}^2}
			+ \mu (g, \mathrm{div}\, v)
			_{{L}_{{\Lambda^{3}}}^2}
			& =
			& \langle f -  g u, v \rangle,
			\\
		u(x,0) 
			& =
			& u_0.
		\end{array}
		\right.
	\end{equation}
	
	\begin{theorem}
		\label{t.exist.first}
		Let $g \in C (I,L^2_{\Lambda^{3}}) \cap L^2 (I,H_{\Lambda^{3}}^1)$ be the solution of \eqref{eq.weak_g} for given any pair $(f,u_0) \in L^2 (I,V^0_{\Lambda^{2}}) \times V^1_{\Lambda^{2}}$. Then there exist a unique differential form
		$u \in C (I,V_{\Lambda^{2}}^1) \cap L^2 (I,V_{\Lambda^{2}}^2)$ 
		satisfying \eqref{eq.NS.lin.weak_u}
		for all $v \in V_{\Lambda^{3}}^k$ with $ k\geq 2 $.
	\end{theorem}
	
	\begin{proof}
	Indeed, let $\{u_m\}$ be the sequence of Faedo-Galerkin approximations (see \eqref{eq36}) such that the sequence $\{g_m\} = \{\mathrm{div}\,u_m\}$ converges to $ g\in C (I,L^2_{\Lambda^{3}}) \cap L^2 (I,H_{\Lambda^{3}}^1) $. Substituting $ u_m $ to \eqref{eq.NS.lin.weak_u} instead of $ u $, $ v $ and integrating by $ t\in [0, t_0]$ we have
	\begin{equation}\label{eq40}
	 \|u_m(\cdot,t)\|^2_{{L}_{{\Lambda^{2}}}^2}
		+ 2 \mu \|g_m\|^2_{L^2(I_{t_0},L^2_{\Lambda^{3}})} = \|u_m(x,0)\|^2_{{L}_{{\Lambda^{2}}}^2} + 2\int_0^{t_0}\langle f -  g_m u_m, u_m \rangle dt.
	\end{equation}
	As usual, we evaluate by H\"older inequality
	\[
	2\left| \int_0^{t_0}\langle f, u_m \rangle dt\right| \leq \int_0^{t_0}\|u_m\|^2_{{L}_{{\Lambda^{2}}}^2}dt+ \int_0^{t_0}\|f\|^2_{{L}_{{\Lambda^{2}}}^2}dt
	\]
	and by Gagliardo-Nirenberg inequality
	\[
	2\left| \int_0^{t_0}\langle g_m u_m, u_m \rangle dt\right| \leq \int_0^{t_0}\|u_m\|^2_{{L}_{{\Lambda^{2}}}^4}\|g_m\|_{{L}_{{\Lambda^{3}}}^2}dt\leq
	\]
	\[
	c\int_0^{t_0} \|g_m\|_{{L}_{{\Lambda^{3}}}^2} \left(  \|g_m\|^{3/2}_{{L}_{{\Lambda^{3}}}^2} \|u_m\|^{1/2}_{{L}_{{\Lambda^{2}}}^2} + \|u_m\|^2_{{L}_{{\Lambda^{2}}}^2}\right) dt\leq
	\]
		\[
		2\mu\int_0^{t_0} \|g_m\|^{2}_{{L}_{{\Lambda^{3}}}^2}dt + 
		c_1\int_0^{t_0} \|g_m\|^{4}_{{L}_{{\Lambda^{3}}}^2} \|u_m\|^{2}_{{L}_{{\Lambda^{2}}}^2} dt +
	c_2 \int_0^{t_0} \|g_m\|_{{L}_{{\Lambda^{3}}}^2} \|u_m\|^2_{{L}_{{\Lambda^{2}}}^2} dt
	\]
	with some positive constants $ c $, $ c_1 $ and $ c_2 $. Then 
	\begin{equation}\label{eq41}
		\|u_m(\cdot,t)\|^2_{{L}_{{\Lambda^{2}}}^2} \leq \|u_0\|^2_{{L}_{{\Lambda^{2}}}^2} +  \int_0^{t_0}\|f\|^2_{{L}_{{\Lambda^{2}}}^2}dt+
		c\int_0^{t_0}\|u_m\|^2_{{L}_{{\Lambda^{2}}}^2}dt
	\end{equation}
	with constant $ c>0 $, independent on $ m $. It follows from Gronwall's Lemma that
	\begin{equation}\label{eq42}
		\|u_m(\cdot,t)\|^2_{{L}_{{\Lambda^{2}}}^2} \leq C,
	\end{equation}
	where constant $ C $ depends on norms $ \|f\|^2_{L^2(I_{t_0}, L^2_{\Lambda^{2}})} $, $ \|u_0\|^2_{{L}_{{\Lambda^{2}}}^2} $ and $ \|g\|_{C(I_{t_0}, L^2_{\Lambda^{2}})} $, but not on $ m $. 
	
	It follows that the sequence $ u_m $ is bounded in $ L^\infty(I_{t_0},L^2_{\Lambda^{3}}) $ and  there is a subsequence that converges weakly-$ * $ in $ L^\infty(I_{t_0},L^2_{\Lambda^{3}}) $ to some $ u\in L^\infty(I_{t_0},L^2_{\Lambda^{3}}) $.  We again use the same designation $ u_m $ for such a subsequence. Under hypothesis of this Theorem the sequence $ g_m = \mathrm{div}\, u_m $ converges to $ g\in C (I,L^2_{\Lambda^{3}}) \cap L^2 (I,H_{\Lambda^{3}}^1) $, then actually $ u\in C (I,V^1_{\Lambda^{3}}) \cap L^2 (I,V_{\Lambda^{3}}^2)$. Passing to the limit in \eqref{eq40} 
	with respect to $ m\to \infty $ we conclude that the element $ u $ satisfies \eqref{eq.NS.lin.weak_u}.
	
Let now $ u' $ and $ u'' $ are two solutions of \eqref{eq.NS.lin.weak_u} such that $ \mathrm{div}\, u' = \mathrm{div}\, u'' =g $. Hence differential form $ u = u'-u'' $ satisfies \eqref{eq.NS.lin.weak_u} with zero date $ (f, u_0) = (0,0)$. It follows from \eqref{eq41} and Gronwall-Perov's Lemma that $ \|g(\cdot, t)\|_{L^2_{\Lambda^{3}}} = 0 $, then the Problem \eqref{eq.NS.lin.weak_u} has unique solution.

Moreover, if $ u_1 $, $ u_2 $ are two solutions of \eqref{eq.NS.lin.weak_u}, corresponding to the solutions $ g_1 = \mathrm{div}\, u_1 $ and $ g_2 = \mathrm{div}\, u_2 $ of \eqref{eq.weak_g}, then differential form $ u = u_1-u_2 $ satisfies 
	\begin{equation}
	\label{eq.NS.lin.weak_u_e}
	\left\{
	\begin{array}{rcl}
		\displaystyle
		\frac{d}{dt} (u,v)_{{L}_{{\Lambda^{2}}}^2}
		+ \mu (g, \mathrm{div}\, v)
		_{{L}_{{\Lambda^{3}}}^2}
		& =
		& \langle -  g u, v \rangle,
		\\
		u(x,0) 
		& =
		& 0,
	\end{array}
	\right.
\end{equation}
where $ g=g_1-g_2 $.
\begin{equation}\label{eq48}
	\|u(\cdot,t)\|^2_{{L}_{{\Lambda^{2}}}^2}
	+ 2 \mu \|g\|^2_{L^2(I_{t_0},L^2_{\Lambda^{3}})} = -2\int_0^{t_0}\langle   g u, u \rangle dt.
\end{equation}
Applying the Gagliardo-Nirenberg inequality we have
\[
2\left|\int_0^{t_0}\langle   g u, u \rangle dt \right| \leq 
c_1\int_0^{t_0}\|g\|_{L^2_{\Lambda^{3}}}\left( \|\nabla u\|^{3/4}_{{L}_{{\Lambda^{2}}}^2} \|u\|^{1/4}_{{L}_{{\Lambda^{2}}}^2} + \|u\|_{{L}_{{\Lambda^{2}}}^2} \right)^2dt\leq
\]
\[ 
c_2\int_0^{t_0}\left( \| g\|^{5/2}_{{L}_{{\Lambda^{2}}}^2} \|u\|^{1/2}_{{L}_{{\Lambda^{2}}}^2} + \|u\|^2_{{L}_{{\Lambda^{2}}}^2} \|g\|_{L^2_{\Lambda^{3}}}\right)dt\leq
\]
\[
2\mu \|g\|^2_{L^2(I_{t_0},L^2_{\Lambda^{3}})} +c_3\left( \|g\|^4_{C(I_{t_0},L^2_{\Lambda^{3}})} + \|g\|_{C(I_{t_0},L^2_{\Lambda^{3}})} \right) \int_0^{t_0} \|u\|_{{L}_{{\Lambda^{2}}}^2} dt
\]
with positive constants $ c_1 $, $ c_2 $ and $ c_3 $. The last inequality, 
\eqref{eq48} and Gronwall-Perov's Lemma yields
\[
\|u(\cdot,t)\|^2_{{L}_{{\Lambda^{2}}}^2} \leq 0,
\]
then $ u_1 = u_2 $ and the Problem \eqref{eq.NS.lin.weak_u} has unique solution.
	
\end{proof}

\begin{corollary}\label{corol}
	Under hypothesis of the Theorem \ref{t.weak.g}, let $ g\in C (I,L^2_{\Lambda^{3}}) \cap L^2 (I,H_{\Lambda^{3}}^1) $ is a solution of \eqref{eq.weak_g} and $ u \in C (I,V_{\Lambda^{2}}^1) \cap L^2 (I,V_{\Lambda^{2}}^2) $ is a corresponding to $ g $ solution of \eqref{eq.NS.lin.weak_u}. If moreover $g\in C (I,H^1_{\Lambda^{3}}) \cap L^2 (I,H_{\Lambda^{3}}^1)  $, then the solution $ g $ is unique.
\end{corollary}
\begin{proof}	
Indeed, let $ g_1, g_2 \in C (I,H^1_{\Lambda^{3}}) \cap L^2 (I,H_{\Lambda^{3}}^1)$ are two solutions of \eqref{eq.weak_g} with corresponding forms $ u_1, u_2 \in C (I,H^1_{\Lambda^{3}}) \cap L^2 (I,H_{\Lambda^{3}}^1) $ satisfying \eqref{eq.NS.lin.weak_u}. Hence differential form $ g = g_1-g_2 $ satisfies 
	\begin{equation}
		\label{eq43}
		\left\{
		\begin{array}{rcl}
			\displaystyle
			\frac{d}{dt} \|g\|^2_{{L}_{{\Lambda^{3}}}^2}
			+ \mu \|\nabla g\|^2_{{L}_{{\Lambda^{2}}}^2}
			& =
			& \left\langle  - (g_1^2 - g_2^2) -\left( \nabla g_1\cdot(\varphi^2\,\nabla g_1 + \Pi^2 u_1)-\right.\right.  \\
			&&\left.  \left.  \nabla g_2\cdot(\varphi^2\,\nabla g_2 + \Pi^2 u_2)\right)  , v \right\rangle ,
			\\
			g (\cdot,0)
			& =
			& 0.
		\end{array}
		\right.
	\end{equation}
	Integrating by $ t\in I_{t_0} $ we get
	\begin{equation}\label{eq44}
		\|g(\cdot,t\|^2_{{L}_{{\Lambda^{3}}}^2}
		+ 2\mu \int_0^{t_0}\|\nabla g\|^2_{{L}_{{\Lambda^{2}}}^2}dt \leq 2\int_0^{t_0}\left| \left\langle  g(g_1 + g_2) +\right. \right. 
	\end{equation}
	\[
	\left. \left. \left( \nabla g_1\cdot(\varphi^2\,\nabla g_1 + \Pi^2 u_1)- \nabla g_2\cdot(\varphi^2\,\nabla g_2 + \Pi^2 u_2)\right)  , v \right\rangle \right| dt ,
	\]
	We have to estimate right side of \eqref{eq44}. First, it follows from Gagliardo-Nirenberg interpolation inequality that
	\begin{equation}\label{eq45}
		2\int_0^{t_0}\left| \left\langle  g(g_1 + g_2), g\right\rangle \right|dt \leq 
	\end{equation}
	\[
	c\|(g_1 + g_2)\|_{C(I_{t_0},L^2_{\Lambda^{3}})}\int_0^{t_0} \left(  \|\nabla g\|^{3/2}_{L^2_{\Lambda^{3}}} \| g\|^{1/2}_{L^2_{\Lambda^{3}}} + \| g\|^{2}_{L^2_{\Lambda^{3}}}\right) dt\leq 
	\]
	\[
	\mu\int_0^{t_0} \|\nabla g\|^2_{L^2_{\Lambda^{3}}} dt + c_1 \int_0^{t_0} \| g\|^{2}_{L^2_{\Lambda^{3}}} dt
	\]
	with positive constants $ c $,  $ c_1 $. Next,
	\begin{equation}\label{eq46}
		2\int_0^{t_0}\left| \left\langle  \nabla g_1\cdot\varphi^2\,\nabla g_1 - \nabla g_2\cdot\varphi^2\,\nabla g_2 + \nabla g_1\cdot\varphi^2\,\nabla g_2 - \nabla g_1\cdot\varphi^2\,\nabla g_2, g\right\rangle \right| dt\leq 
	\end{equation}
	\[
	2\int_0^{t_0}\left| \left\langle  \nabla g_1\cdot\varphi^2\,\nabla g, g\right\rangle \right| dt + 2\int_0^{t_0}\left| \left\langle  \nabla g\cdot\varphi^2\,\nabla g_2, g\right\rangle \right| dt\leq
	\]
	\[
	c_1\|g_1\|_{C(I_{t_0},H^1_{\Lambda^{3}})} \int_0^{t_0}
	\left(\|\nabla g\|^{3/4}_{L^2_{\Lambda^{3}}}\|g\|^{5/4}_{L^2_{\Lambda^{3}}} +\|g\|^{2}_{L^2_{\Lambda^{3}}}\right)dt  + 
	\]
	\[
	c_2\|g_1\|_{C(I_{t_0},H^1_{\Lambda^{3}})} \int_0^{t_0}
	\left(\|\nabla g\|^{7/4}_{L^2_{\Lambda^{3}}}\|g\|^{1/4}_{L^2_{\Lambda^{3}}} +\|g\|^{2}_{L^2_{\Lambda^{3}}}\right)dt \leq
	\]
	\[
	\mu\int_0^{t_0} \|\nabla g\|^2_{L^2_{\Lambda^{3}}} dt + c \int_0^{t_0} \| g\|^{2}_{L^2_{\Lambda^{3}}} dt
	\]
	with some positive constants $ c $, $ c_1 $ and $ c_2 $. Finally,
	\begin{equation}\label{eq47}
		2\int_0^{t_0}\left| \left\langle  \nabla g_1\cdot\Pi^2 u_1 - \nabla g_2\cdot\Pi^2 u_2 + \nabla g_2\cdot\Pi^2 u_1 - \nabla g_2\cdot\Pi^2 u_1, g\right\rangle \right|dt \leq 
	\end{equation}
	\[
	2\int_0^{t_0}\left| \left\langle  \nabla g\cdot\Pi^2 u_1, g\right\rangle \right|dt + 2\int_0^{t_0}\left| \left\langle  \nabla g_2\cdot\Pi^2 (u_1-u_2), g\right\rangle \right|dt
	\]
	The Theorem \ref{t.exist.first} implies that $ u_1=u_2 $. On the other hand, integrating by parts we easily see that
	\[
	2\int_0^{t_0}\left| \left\langle  \nabla g\cdot\Pi^2 u_1, g\right\rangle \right|dt=0,
	\]
	and then \eqref{eq47} equals to zero.
	
	Finally, using by \eqref{eq44} - \eqref{eq47} we get
	\[
	\|g(\cdot,t\|^2_{{L}_{{\Lambda^{3}}}^2}\leq c \int_0^{t_0} \| g\|^{2}_{L^2_{\Lambda^{3}}} dt
	\]
	with some constant $ c>0 $. Then it follows from Gronwall-Perov's Lemma that $ \|g(\cdot, t)\|_{L^2_{\Lambda^{3}}} = 0 $, then the Problem \eqref{eq.weak_g} has unique solution.
	\end{proof}

	\begin{theorem}
		\label{t.exist.g}
		Let 
		$s \in \mathbb N$ and $k \in {\mathbb Z}_+$ with $k>3/2$. Then for all 
		\[
		(f,u_0)\in B_{\Lambda^{2},\mathrm{for}}^{k+1,2(s-1),s-1}(X_T) \times V_{\Lambda^{2}}^{2s+k+1}
		\]
		there exist a time $ T_k\in (0,T] $ such that the Problem \eqref{eq38} has solution
		\[
		g\in B_{\Lambda^{3},\mathrm{for}}^{k,2s,s}(X_{T_k}).
		\]
Moreover, solution $ g $ is unique, if form $ u $ in \eqref{eq37} satisfied \eqref{eq.NS.lin.weak_u}.
	\end{theorem}
\begin{proof}
	First of all, denote by
	\begin{equation*}
		\Lambda_{r} = 
		\begin{cases}
			\Lambda^{3},\quad &	r\ \text{is even},\\
			\Lambda^{2},\quad &	r\ \text{is odd}.\\
		\end{cases}
	\end{equation*}
	
	As before, let $ g_m $ be the Faedo-Galerkin approximations, see \eqref{faedo}.
	We start with the following priory estimates.
		\begin{lemma}
		\label{l.OM.bound}
		Under hypothesis of the Theorem \ref{t.exist.g}, 
		if $(f,u_0) \in B^{k+1,0,0}_{{{2}},\mathrm{for}}(X_T) \times V^{k+3}_{\Lambda^{2}}$ with some $ k\in {\mathbb Z}_+ $, then there exist a time $ T_{k} \in (0,T]$ such that 
		\begin{equation}\label{povish}
			\| \nabla^{k'}_2 g_m \|^2_{C (I_{T_{k}},{L}^2_{\Lambda_{k'}}   )}
			+ \mu\, \| \nabla^{k'+1}_2 g_m \|^2_{L^2 (I_{T_{k}},{L}^2_{\Lambda_{k'+1}} )}
			\leq
			C_{k'}
		\end{equation}
		for any $0 \leq k' \leq k+2$, where $ I_{T_{k}} = [0,{T_{k}}] $ and the constants
		$C_{k'} = C_{k'} (\mu,f,u_0) > 0$
		depending on
		$k'$,
		$\mu$
		and the norms
		$\| f \|_{B^{k+1,0,0}_{2,\mathrm{for}} (X_{T_{k}})}$,
		$\| u_0 \|_{V^{k+3}_{\Lambda^{2}}}$
		but not on $m$.
	\end{lemma}
	
	\begin{proof}
	Indeed, if $ k'=0 $ then \eqref{povish} follows immediately from \eqref{eq_edin_11} and Gronwall-Perov's Lemma. Now, substituting $ g_m $ and $ \nabla_{3}^{2r} g_m$ in \eqref{eq38} instead of $ g $ and $ v $ respectively with some $ r\in \mathbb{N} $ and integrating by $ t \in[0,T]$ we get
		\begin{equation}\label{eq23}
			\| \nabla_{3}^r g_m(\cdot, t)\|^2_{{L}_{\Lambda_{r}}^2}
			+ 2\mu\int_0^t \|\nabla_{3}^{r+1} g_m\|^2_{{L}_{{\Lambda^{3}}}^2}dt
			=
		\end{equation}
	\[
\|\nabla_{3}^r g_m(\cdot, 0 )\|^2_{{L}_{\Lambda_{r}}^2} + 	2\int_0^t \langle \mathrm{div}\,f - g_m^2 -\nabla g_m\cdot\varphi^2\,\nabla g_m - \nabla g_m \cdot\Pi^2 u_m, \nabla_{3}^{2r} g_m \rangle dt.
\]
	We have to estimate the right side of  \eqref{eq23}. First,
		\begin{equation}\label{eq_edin_4_1}
			2	\left| \int_0^t \langle \mathrm{div}\,f, \nabla_{3}^{2r} g_m \rangle dt \right| \leq  2\int_0^t 
			\| \nabla_{3}^{r-1}\mathrm{div}\,f\|_{{L}^2_{\Lambda_{r-1}}} 
			\| \nabla_{3}^{r+1} g_m \|_{{L}^2_{\Lambda_{r+1}}} dt \leq
		\end{equation}
		\[
		\frac{4}{\mu}\int_0^t
		\| \nabla_{3}^{r-1}\mathrm{div}\,f\|^2_{{L}^2_{\Lambda_{r-1}}}dt + \frac{\mu}{4} \int_0^t \| \nabla_{3}^{r+1} g_m \|^2_{{L}^2_{\Lambda_{r+1}}} dt.
		\]
		Further,
		\begin{equation}\label{eq_edin_5_1}
			2	\left| \int_0^t \langle g_m^2, \nabla_{3}^{2r} g_m \rangle dt\right| \leq 
			2	\int_0^t \|\nabla_{3}^{r-1}( g_m^2)\|_{L^2_{\Lambda_{r-1}}} \|\nabla_{3}^{r+1} g_m\|_{L^2_{\Lambda_{r+1}}}
			dt.
		\end{equation}
		Let $ r\geq2 $, using by H\"older and Gagliardo-Nirenberg
		inequalities 
		we get
		\begin{equation}\label{eq24}
			\|\nabla_{3}^{r-1}( g_m^2)\|_{L^2_{\Lambda_{r-1}}} \leq \sum_{|\alpha|+|\beta| =r-1}c_{\alpha\beta}\|\partial^\alpha g_m\|_{L^4_{\Lambda^3}}\|\partial^\beta g_m\|_{L^4_{\Lambda^3}}\leq
		\end{equation}
		\[
		\sum_{|\alpha|+|\beta| =r-1}c_{\alpha\beta} \left(  \left( \|\nabla_{3}^{|\alpha|+1} g_m\|_{L^2_{\Lambda_{|\alpha|+1}}} + \|\nabla_{3}^{|\alpha|} g_m\|_{L^2_{\Lambda_{|\alpha|}}} \right)^{3/4} \|\nabla_{3}^{|\alpha|} g_m\|_{L^2_{\Lambda_{|\alpha|}}}^{1/4} +\right. 
		\]
		\[
		\left. + \|\nabla_{3}^{|\alpha|} g_m\|_{L^2_{\Lambda_{|\alpha|}}}\right)
		\left(  \left( \|\nabla_{3}^{|\beta|+1} g_m\|_{L^2_{\Lambda_{|\beta|+1}}} + \|\nabla_{3}^{|\beta|} g_m\|_{L^2_{\Lambda_{|\beta|}}} \right)^{3/4} \|\nabla_{3}^{|\beta|} g_m\|_{L^2_{\Lambda_{|\beta|}}}^{1/4} +\right. 
		\]
		\[
		\left. + \|\nabla_{3}^{|\beta|} g_m\|_{L^2_{\Lambda_{|\beta|}}}\right) \leq 
		c\left(  \|g_m\|^{2}_{H^{r-1}_{\Lambda^3}} + \|g_m\|^{5/4}_{H^{r-1}_{\Lambda^3}} \|\nabla_{3}^{r} g_m\|^{3/4}_{L^2_{\Lambda_{r}}}\right)
		\]
		with some positive constants $ c $ and $ c_{\alpha\beta} $. For the exception case $ r=1 $ the last inequality take the form
		\begin{equation}\label{eq26}
			\|\nabla_{3}( g_m^2)\|_{L^2_{\Lambda_{r-1}}} \leq
			c\left(  \|g_m\|^{2}_{L^2_{\Lambda^3}} + \|g_m\|^{1/2}_{L^2_{\Lambda^3}} \|\nabla_{3} g_m\|^{3/2}_{L^2_{\Lambda^2}}\right)
		\end{equation}
	because of in \eqref{eq24} arise a case when $ |\alpha| = |\beta| = r-1 = 0$.
		It follows from \eqref{eq_edin_5_1}, \eqref{eq24} and Young's inequality that
		\begin{equation}\label{eq25}
				2	\left| \int_0^t \langle g_m^2, \nabla_{3}^{2r} g_m \rangle dt\right| \leq \frac{\mu}{4}\int_0^t \|\nabla_{3}^{r+1} g_m\|_{L^2_{\Lambda_{r+1}}}^2 dt + 
		\end{equation}
		\[
		c\|g_m\|^{4}_{C(I,H^{r-1}_{\Lambda^3})} +c\|g_m\|^{5/2}_{C(I,H^{r-1}_{\Lambda^3})}\int_0^t \|\nabla_{3}^{r} g_m\|^{3/2}_{L^2_{\Lambda_{r}}}dt 
		\]
		for $ r\geq 2 $ and
		\begin{equation}\label{eq25_1}
			2	\left| \int_0^t \langle g_m^2, \nabla_{3}^{2} g_m \rangle dt\right| \leq \frac{\mu}{4}\int_0^t \|\nabla_{3}^{2} g_m\|_{L^2_{\Lambda^3}}^2 dt + 
		\end{equation}
			\[
		 c\|g_m\|^{4}_{C(I,L^2_{\Lambda^2})} + c\|g_m\|_{C(I,L^2_{\Lambda^2})}\int_0^t \|\nabla_{3} g_m\|^{3}_{L^2_{\Lambda^2}}dt
		\]
		for $ r =1 $ with some constant $ c>0 $. 
		
		Next,
		\begin{equation}\label{eq27}
				2\left| \int_0^t \langle\nabla g_m\cdot\varphi^2\,\nabla g_m , \nabla_{3}^{2r} g_m \rangle dt\right| \leq 
			\end{equation}
				\[
					2	\int_0^t 
				\|\nabla_{3}^{r-1}\left( \nabla g_m\cdot\varphi^2\,\nabla g_m\right)\|_{L^2_{\Lambda_{r-1}}} \|\nabla_{3}^{r+1} g_m\|_{L^2_{\Lambda_{r+1}}}dt.
				\]
	Analogous by \eqref{eq24} we have
		\begin{equation}\label{eq29}
			\|\nabla_{3}^{r-1}\left( \nabla g_m\cdot\varphi^2\,\nabla g_m\right)\|_{L^2_{\Lambda_{r-1}}} \leq
		\end{equation}
\[
c\left( \|g_m\|_{H^{r-1}_{\Lambda^3}}\|\varphi^2 g_m\|_{H^{r+1}_{\Lambda^3}} + \|g_m\|^{1/4}_{H^{r-1}_{\Lambda^3}}\|\varphi^2 g_m\|_{H^{r+1}_{\Lambda^3}} \|\nabla_{3}^{r} g_m\|^{3/4}_{L^2_{\Lambda_{r}}} + \right. 
\]
\[
\left. \|\nabla_{3}^{r} g_m\|^{1/4}_{L^2_{\Lambda_{r}}}\|\varphi^2 g_m\|_{H^{r+1}_{\Lambda^3}} \|\nabla_{3}^{r+1} g_m\|^{3/4}_{L^2_{\Lambda_{r}}} \right) 
\]
with $ r\in \mathbb{N} $ and some constant $ c>0 $. Theorem \ref{parametrix.laplas} imply that $ \|\varphi^2 g_m\|_{H^{r+1}_{\Lambda^3}}\leq c\| g_m\|_{H^{r-1}_{\Lambda^3}} $ with some positive constant $ c $, then 
\begin{equation}\label{eq30}
		2\left| \int_0^t \langle\nabla g_m\cdot\varphi^2\,\nabla g_m , \nabla_{3}^{2r} g_m \rangle dt\right| \leq \frac{\mu}{4}\int_0^t \|\nabla_{3}^{r+1} g_m\|_{L^2_{\Lambda_{r+1}}}^2 dt + 
\end{equation}
\[
c\|g_m\|^{4}_{C(I,H^{r-1}_{\Lambda^3})} + c\|g_m\|^{5/2}_{C(I,H^{r-1}_{\Lambda^3})}\int_0^t \|\nabla_{3}^{r} g_m\|^{3/2}_{L^2_{\Lambda_{r}}}dt + 
\]
\[
c\|g_m\|^{10}_{C(I,H^{r-1}_{\Lambda^3})}\int_0^t\|\nabla_{3}^{r} g_m\|^{2}_{L^2_{\Lambda_{r}}}dt
\]
with $ c>0 $.

Finally,
\begin{equation}\label{eq28}
	2\left| \int_0^t \langle\nabla g_m \Pi^2 u_m, \nabla_{3}^{2r} g_m \rangle dt\right| \leq
\end{equation}
\[
2	\int_0^t \|\nabla_{3}^{r+1} g_m\|_{L^2_{\Lambda_{r+1}}}
\|\nabla_{3}^{r-1}\left( \nabla g_m\cdot \Pi^2 u_m\right)\|_{L^2_{\Lambda_{r-1}}}dt,
\]
	and we have again
	\begin{equation}\label{eq31}
		\|\nabla_{3}^{r-1}\left( \nabla g_m\cdot \Pi^2 u_m\right)\|_{L^2_{\Lambda_{r-1}}}\leq
	\end{equation}	
		\[
		c\left( \|g_m\|_{H^{r-1}_{\Lambda^3}}\|\Pi^2 u_m\|_{H^{r}_{\Lambda^2}} + \|g_m\|^{1/4}_{H^{r-1}_{\Lambda^3}}\|\Pi^2 u_m\|_{H^{r}_{\Lambda^2}} \|\nabla_{3}^{r} g_m\|^{3/4}_{L^2_{\Lambda_{r}}} + \right. 
		\]
		\[
		\left. \|\nabla_{3}^{r} g_m\|^{1/4}_{L^2_{\Lambda_{r}}}\|\Pi^2 u_m \|_{H^{r}_{\Lambda^3}} \|\nabla_{3}^{r+1} g_m\|^{3/4}_{L^2_{\Lambda_{r}}} \right) 
		\]
	with positive constant $ c $. Operator $ \Pi^2 $ is bounded in $ L^2_{\Lambda^2} $ by the Hodge Theorem \ref{parametrix.laplas}. On the other hand Theorem \ref{t.exist.first} yields that the sequence $ \{u_m\} $ is bounded in $ L^2_{\Lambda^2} $ (see \eqref{eq42}), then $ \|\Pi^2 u_m\|_{H^{r}_{\Lambda^2}}\leq c \| g_m\|_{H^{r-1}_{\Lambda^3}}$ and we get
		
	\begin{equation}\label{eq49}
		2\left| \int_0^t \langle\nabla g_m \Pi^2 u_m, \nabla_{3}^{2r} g_m \rangle dt\right| \leq  
		\frac{\mu}{4}\int_0^t \|\nabla_{3}^{r+1} g_m\|_{L^2_{\Lambda_{r+1}}}^2 dt + 
	\end{equation}
	\[
	c\|g_m\|^{4}_{C(I,H^{r-1}_{\Lambda^3})} + c\|g_m\|^{5/2}_{C(I,H^{r-1}_{\Lambda^3})}\int_0^t \|\nabla_{3}^{r} g_m\|^{3/2}_{L^2_{\Lambda_{r}}}dt + 
	\]
	\[
	c\|g_m\|^{10}_{C(I,H^{r-1}_{\Lambda^3})}\int_0^t\|\nabla_{3}^{r} g_m\|^{2}_{L^2_{\Lambda_{r}}}dt
	\]
	with $ c>0 $.	
		
		It follows from \eqref{eq23} - \eqref{eq49} and Gronwall-Perov's Lemma that if $(f,u_0) \in B^{k+1,0,0}_{{{2}},\mathrm{for}}(X_T) \times V^{k+3}_{\Lambda^{2}}$ and the norm $ \|g_m\|_{C(I,H^{r-1}_{\Lambda^3})} $ is bounded for some $ r \in \mathbb{N}$, $ r\leq k+2 $, then  there exist a time $ t_r\in (0, t_0] $ and a positive constant $ C_r $, depending on the norms $\| f \|_{B^{r+1,0,0}_{2,\mathrm{for}} (X_{T_{k}})}$ and
		$\| u_0 \|_{V^{r+3}_{\Lambda^{2}}}$, such that
		\begin{equation}\label{eq50}
				\| \nabla_{3}^r g_m(\cdot, t)\|^2_{{L}_{\Lambda_{r}}^2}
				+ \mu\int_0^{t_r} \|\nabla_{3}^{r+1}g_m\|^2_{{L}_{{\Lambda^{3}}}^2}dt
				\leq C_r( \mu,f,u_0).
		\end{equation}
	
	Using by \eqref{eq50} consistently for $ r=1,\dots,k+2 $ we get family of times $ t_r $. Denote $ T_k = \min\limits_{r\leq k+2} t_r$,
	then \eqref{eq50} yields that for any $ k\in {\mathbb Z}_+ $ there exist a time $ T_k  $ such that \eqref{l.OM.bound} is fulfilled.
\end{proof}
	
	Theorem \ref{t.weak.g} imply that there exist a solution $g \in C (I,L^2_{\Lambda^{3}}) \cap L^2 (I,H_{\Lambda^{3}}^1)$ of \eqref{eq.weak_g}. On the other hand, it follows from Lemma \ref{l.OM.bound} that for each $	(f,u_0)\in B_{\Lambda^{2},\mathrm{for}}^{k+1,2(s-1),s-1}(X_T) \times V_{\Lambda^{2}}^{2s+k+1}$ there exist a time $ T_{k} \in (0,T]$ and a subsequence $\{ g_{m'} = \mathrm{div}\, u_{m'}\}$ such that $\{ g_{m'}\}$ converges weakly in $L^2 (I_{T_{k}},{L}^{2}_{\Lambda^3} )$ and $^{\ast}$-weakly in $L^\infty (I_{T_{k}},{H}^{k+2}_{\Lambda^3}  ) \cap L^2 (I,{H}^{k+3}_{\Lambda^3}  )$ to an element $g$, then $ g\in B_{\Lambda^{3},\mathrm{for}}^{k,2s,s}(X_{T_k}) $. Moreover, the uniqueness of $g$ immediately follows from Corollary \ref{corol}.

\end{proof}

\begin{theorem}
	\label{t.exist.u}
	Let 
	$s \in \mathbb N$ and $k \in {\mathbb Z}_+$ with $k\geq 2$. Then for all 
	\[
	(f,u_0)\in B_{\Lambda^{2},\mathrm{for}}^{k+1,2(s-1),s-1}(X_T) \times V_{\Lambda^{2}}^{2s+k+1}
	\]
	there exist a time $ T^*\in (0,T] $ such that the Problem \eqref{eq34} has unique solution
	\[
	(u,p)\in B_{\Lambda^{2},\mathrm{vel}}^{k+1,2s,s}(X_{T_k})\times B_{\Lambda^{2},\mathrm{pre}}^{k+2,2(s-1),s-1}(X_{T_k}).
	\]
\end{theorem}
\begin{proof}
Indeed, apply the projection $ P^2 $ (see Lemma \ref{proector} above) to the equation \eqref{eq34} we have
	\begin{equation}\label{eq51}
	\begin{cases}
		\partial_t u + \mu \Delta_{2} u + P^2 N^{2}(u) = P^2 f& \text{in } X \times (0,T),\\
		u(x,0) = u_0& \text{in } X,
	\end{cases}
\end{equation}
then the form $ p $ actually has to satisfy the equation 
\begin{equation}\label{eq52}
	\mathrm{rot}\, p = (I -P^2)( f - N^{2}(u))\quad \text{in } X \times (0,T).
\end{equation}

Multiplying \eqref{eq51} by $v \in V_{\Lambda^{3}}^k$ we get the Problem \eqref{eq.NS.lin.weak_u}, then the existence and regularity of solution $ u $ follows immediately from the Theorems \ref{t.exist.first} and \ref{t.exist.g}.
On the other hand, it follows from Lemma \ref{p.nabla.Bochner} that there exist unique differential form $ p\in B_{\Lambda^{2},\mathrm{pre}}^{k+2,2(s-1),s-1}(X_{T_k}) $, satisfying \eqref{eq52}.

\end{proof}

\bigskip

\textit{The work was supported by the Foundation for the Advancement of Theoretical Physics and Mathematics "BASIS".}


\begin{thebibliography}{}
	
	\bibitem{ADN59}
	S.	Agmon, A. Douglis, L. Nirenberg,
	Estimates near the boundary for solutions of elliptic partial differential equations satisfying general boundary conditions, Part 1.  Comm. Pure Appl. Math. \textbf{12}, 623--727 (1959)
	
	\bibitem{Agronovich94}
	M.S. Agranovich, Elliptic Operators on Closed Manifolds, Partial Differential Equations VI. Encyclopedia of Mathematical Sciences, vol 63. Springer, Berlin, Heidelberg, (1994)
	
	\bibitem{Agronovich15}
	M. Agranovich,	Sobolev Spaces, Their Generalizations and Elliptic Problems in Smooth and Lipschitz Domains, Springer, (2015)
	
	\bibitem{Aubin82}
	T. Aubin,	Nonlinear Analysis on Manifolds. Monge-Ampire Equations, Springer-Verlag New York, (1982)
	
	\bibitem{Bourbaki}
	N. Bourbaki, Topological vector spaces,
	Elements of mathematics. Springer-Verlag, Berlin (1987)
	
	\bibitem{Ebin70}
	D.G. Ebin, and J. Marsden, Groups of Diffeormophisms and the motion of an
	incompressible fluid, Annals of Math. 92, 102-163 (1970)
	
	\bibitem{Eidelman}	
	S.D. Eidelman, 
	Parabolic equations, Partial differential equations 6, Itogi Nauki i Tekhniki. Ser. Sovrem. Probl. Mat. Fund. Napr., 63, VINITI, Moscow, 201-313 (1990)
	
	\bibitem{Fillipov}
	A.F. Filippov,
	Differential equations with discontinuous right-hand side, 225.
	Nauka, Moscow (1985)
	
	\bibitem{Gaevsky}
	H.	Gaevsky,  K. Greger, K. Zaharias,
	Nonlinear Operator Equations and Operator Differential Equations, 335.
	Mir, Moscow (1978)
	
	\bibitem{Groe19}
	T. H.	Gronwall,
	Note on the derivatives with respect to a parameter of the solutions of a
	system of differential equations, Ann. of Math., 20:2, 292--296 (1919)
	
	\bibitem{Ham82}
	Hamilton, R. S.,
	The inverse function theorem of Nash and Moser.
	Bull. of the AMS \textbf{7}, no.~1, 65--222, (1982)
	
	\bibitem{LadSoUr67}
	O. A.	Ladyzhenskaya, V. A. Solonnikov, and Ural'tseva, N. N., Linear and Quasilinear Equations of Parabolic Type, 648. Nauka, Moscow (1967)
	
	\bibitem{Ladyz}
	O.A. Ladyzhenskaya, Mathematical theory of viscous incompressible flow. Gordon and Breach, New York, (1969)
	
	\bibitem{Lion69}
	J.-L. Lions, Quelques m\'{e}thodes de r\'{e}solution des prob
	l\`{e}mes aux limites non lin\'{e}are, 588. Dunod/Gauthier-Villars, Paris, (1969)
	
	\bibitem{LiMa72}
	J. L.	Lions, E. Magenes,
	Non-Homogeneous Boundary Value Problems and Applications, 360.
	Springer-Verlag, Berlin et al. (1972)
	
	\bibitem{Tarkhanov19}
	A. Mera, A. Shlapunov, N. Tarkhanov, Navier-Stokes equations for elliptic complexes, J. Sib. Fed. Univ. Math. 
	Phys., 12:1, 3-27 (2019)
	
	\bibitem{Mikh76}
	V. P. Mikhailov,
	Partial Differential Equations, 392.
	Nauka, Moscow (1976)
	
	\bibitem{MPF91}
	D. S.	Mitrinovi\'c, J.E. Pe$\check{c}$ari\'c, A.M. Fink,
	Inequalities Involving Functions and Their Integrals and Derivatives, 587.
	Mathematics and its Applications (East European Series), V. 53, 
	Kluwer Ac. Publ., Dordrecht, 
	Springer-Science + Business Media B.V., Dordrecht (1991)
	
	\bibitem{Nicolaescu}
	L.I. Nicolaescu, Lectures on the Geometry of Manifolds, World Scientific, London, (2007)
	
	\bibitem{Nir59}
	L. Nirenberg, On Elliptic partial differential equations,
	Ann. Sc. Norm. Sup. di Pisa, Cl. Sci., Ser. 3(13), 115--162 (1959)
	
	\bibitem{ShlapunovPar}	
	A.A. Parfenov, A.A. Shlapunov, 
	On the Fredholm property for the steady Navier-Stokes equations in weighted H\"older spaces, Journal of Siberian Federal University, Math. and Phys., N. 11(5), 659-662 (2018)
	
	\bibitem{Sche60}
	M.	Schechter,
	{Negative norms and boundary problems},
	Ann. Math. \textbf{72}, No.~3, 581--593 (1960)
	
	\bibitem{Shl22}
	A.A. Shlapunov,
	Spectral decomposition of Green's integrals and existence of $W^{s,2}\,$-solutions of 
	matrix factorizations of the Laplace operator in a ball, Rend. Sem. Mat. Univ. Padova, 96, 
	237-256 (1996)
	
	\bibitem{ShT20}	
	A. Shlapunov, N. Tarkhanov, Inverse image of precompact sets and existence theorems for the Navier-Stokes equations in spatially periodic setting. https://arxiv.org/abs/2106.07515	
	
	\bibitem{Tarkhanov95}
	N. Tarkhanov, Complexes of Differential Operators,
	Kluwer Academic Publishers, Dordrecht, NL (1995)
	
	\bibitem{Taylor}
	M. Taylor, 
	Pseudodifferential Operators,
	Priceton, New Jersey, Priceton Univerity Press, (1981)
	
	\bibitem{Temam79}
	R.	Temam,
	Navier-Stokes equations: Theory and Numerical Analysis, 408.
	Studies in Math. and its Appl. \textbf{2} (1979).
	
	\bibitem{Temam95}
	R. Temam, Navier-Stokes Equations and Nonlinear Functional Analysis,
	2nd ed.,  SIAM, Philadelphia (1995)
	
	\bibitem{N28B}
	Polkovnikov, A.N., An open mapping theorem for nonlinear operator equations associated with elliptic  complexes, Applicable Analysis, 2022, https://doi.org/10.1080/
	
	\bibitem{ShTar21}
	A.A. Shlapunov, N. Tarkhanov, An open mapping theorem for the Navier-Stokes type equations associated with the de Rham complex over $ \mathbb{R}^n $, Sib. Elektron. Mat. Izv., 18:2 (2021), 1433-1466
	
\end{thebibliography}
\end{document}